\documentclass[oneside]{amsart}

\usepackage[latin1]{inputenc}
\usepackage[T1]{fontenc}
\usepackage{amstext}
\usepackage{amsmath,amssymb,epsf,epsfig}
\usepackage[english]{babel}
\usepackage{amsfonts}
\newtheorem {theo}{Theorem}[section]

\newtheorem{lem}[theo]{Lemma}
\newtheorem{prop}[theo]{Proposition}

\newcommand{\comment}[1]{}
\newcommand\beq{\begin{equation}}
\newcommand\eeq{\end{equation}}
\newcommand\beqa{\begin{eqnarray}}
\newcommand\eeqa{\end{eqnarray}}
\numberwithin{equation}{section}

\newcommand{\R}{\mathbb{R}}

\newcommand{\N}{\mathbb{N}}

\newcommand{\Z}{\mathbb{Z}}

\newcommand\kB{\mathcal{B}}
\newcommand\kC{\mathcal{C}}
\newcommand\kD{\mathcal{D}}
\newcommand\kE{\mathcal{E}}

\newcommand\kQ{\mathcal{Q}}
\newcommand\kF{\mathcal{F}}
\newcommand\kG{\mathcal{G}}
\newcommand\kH{\mathcal{H}}

\newcommand\kN{\mathcal{N}}

\newcommand\kW{\mathcal W}
\newcommand\kV{\mathcal{V}}

\def\la{\lambda}

\def\Ga{\Gamma}
\def\De{\Delta}
\def\La{\Lambda}
\def\Th{\Theta}
\def\th{\theta}

\def\l{\ell}
\def\ep{\varepsilon}
\def\al{\alpha}
\def\ga{\gamma}

\def\wt{\widetilde}
\def\wh{\widehat}

\title{Random walk in a high density\\
dynamic random environment}

\author{Frank den Hollander}
\address{Mathematical Institute, Leiden University, P.O.\ Box 9512,
2300 RA Leiden, The Netherlands}
\email{denholla@math.leidenuniv.nl}

\author{Harry Kesten}
\address{Malott Hall, Cornell University, Ithaca, NY.,14853, USA}
\email{hk21@cornell.edu}

\author{Vladas Sidoravicius}
\address{IMPA, Estrada Dona Castorina 110, Jardim Botanico, Cep 22460-320, Rio de Janeiro, RJ, Brasil}
\email{vladas@impa.br}


\begin{document}

\begin{abstract}
The goal of this note is to prove a law of large numbers for the empirical 
speed of a green particle that performs a random walk on top of a field of 
red particles which themselves perform independent simple random walks on $\Z^d$, 
$d \geq 1$. The red particles jump at rate 1 and are in a Poisson equilibrium 
with density $\mu$. The green particle also jumps at rate 1, but uses different 
transition kernels $p'$ and $p''$ depending on whether it sees a red particle 
or not. It is shown that, in the limit as $\mu\to\infty$, the speed of the 
green particle tends to the average jump under $p'$. This result is far from 
surprising, but it is non-trivial to prove. The proof that is given in this 
note is based on techniques that were developed in \cite{KeSi} to deal with 
spread-of-infection models. The main difficulty is that, due to particle 
conservation, space-time correlations in the field of red particles decay 
slowly. This places the problem in a class of random walks in dynamic 
random environments for which scaling laws are hard to obtain.
\end{abstract}

\maketitle

%%%%%%%%%% SECTION %%%%%%%%%%%%%%%%%%%%%%%%%%%%%%%%%%%%%

\section{Introduction and background}
\label{S1}

%%%%%%%%%%%%%%%%%%%%%%

\subsection{Model and main theorem}
\label{S1.1}

We consider a \emph{green} particle that performs a continuous-time 
random walk on $\Z^d$, $d \geq 1$, under the influence of a field of 
\emph{red} particles which themselves perform independent continuous-time 
simple random walks jumping at rate 1, constituting a dynamic random 
environment. The latter is denoted by 
\beq
\label{Ndef}
N = (N(t))_{t\geq 0} \quad \text{ with } \quad 
N(t)=\{N(x,t) \colon\,x\in\Z^d\},
\eeq
where $N(x,t)\in\N_0=\N\cup\{0\}$ is the number of red particles at 
site $x$ at time $t$. As initial state we take $N(0)=\{N(x,0)\colon\,
x\in\Z^d\}$ to be i.i.d.\ Poisson random variables with mean $\mu$. As 
is well known, this makes $N$ invariant under translations in space 
and time.

Also the green particle jumps at rate 1, however, our assumption is 
that the jump is drawn from two different random walk transition kernels 
$p'$ and $p''$ on $\Z^d$ depending on whether the space-time point of 
the jump is occupied by a red particle or not. We assume that $p'$ and 
$p''$ have finite range, and write 
\beq
\label{vdefs}
v'=\sum_{x\in\Z^d} xp'(0,x), \qquad
v''=\sum_{x\in\Z^d} xp''(0,x),
\eeq 
to denote their mean. We write 
\beq
\label{Gdef}
\kG= (\kG(t))_{t\geq 0}
\eeq
to denote the path of the green particle with $\kG(0)=0$, and write 
$P^\mu$ to denote the joint law of $N$ and $\kG$. Our main result is 
the following asymptotic weak law of large numbers ($\|\cdot\|$ is the 
Euclidean norm on $\R^d$).

\begin{theo}
\label{thm1}
For every $\ep>0$,
\beq
\lim_{\mu\to\infty} \limsup_{t\to\infty} 
P^\mu\{\|t^{-1}\kG(t)-v'\|>\ep\} = 0.
\eeq
\end{theo}

%%%%%%%%%%%%%%%%%%%

\subsection{Discussion}
\label{S1.2}

The result in Theorem~\ref{thm1} is far from surprising. As $\mu\to\infty$, 
at any given time the fraction of \emph{sites} occupied by red particles 
tends to 1. Therefore we may expect that the fraction of \emph{time} the 
green particle sees a red particle tends to 1 also. Consequently, we may 
expect the green particle to almost satisfy a weak law of large numbers 
corresponding to the transition kernel $p'$, as if it were seeing a red 
particle always. Despite this simple intuition, the result in Theorem~\ref{thm1} 
seems non-trivial to prove. The proof in the present note relies on techniques 
developed in \cite{KeSi} to deal with spread-of-infection models. 

The key problem is to show that for large $\mu$ the green particle is unlikely 
to spend an appreciable amount of time in the rare space-time holes of the 
field of red particles. To see why this is non-trivial, consider the case 
$d=1$ with two nearest-neighbor transition kernels $p'$ and $p''$ of the form
\beq
\label{nnchoice}
p'(0,1)=p=p''(0,-1), \qquad p'(0,-1)=1-p=p''(0,1), \qquad p \in (\tfrac12,1),
\eeq
for which $v'=2p-1=-v''>0$. Then the green particle drifts to the right when 
it sees a red particle, but drifts to the left when it sees a hole. Thus, it 
has a tendency to linger around the boundaries of the red clusters, hopping 
in and out of these clusters repeatedly. To prove Theorem~\ref{thm1}, we must 
show that the green particle does not do this in-out hopping too often. The proof in 
the present note uses a \emph{multi-scale renormalization} argument, working 
with ``good'' blocks (where the green particle sees only red clusters) and 
``bad'' blocks (where it also sees some holes). These blocks live on successive 
space-time scales. Estimates on how often the green particle visits the bad 
blocks must be uniform in the path of the green particle and must be sharp in
the limit as $\mu\to\infty$.

%%%%%%%%%%%%%%%%%%%%%%%%%

\subsection{Literature}
\label{S1.3}

How does Theorem~\ref{thm1} relate to the existing literature? So far, random 
walks in three classes of dynamic random environments have been considered:
(1) \emph{independent in time}: globally updated at each unit of time; (2) 
\emph{independent in space}: locally updated according to independent 
single-site Markov chains; (3) \emph{dependent in space and time}. 
Typically, the jumps of the walk are chosen to depend on the environment 
in some local manner. Most papers require additional assumptions on the 
environment, like a strong decay of space-time correlations (see e.g.\ \cite{BaZe},
\cite{DoKeLi}, \cite{ReVo}) or a weak influence on the walk (see e.g.\ \cite{AvdHoRe2}). 
In the latter case the random walk in dynamic random environment is a small 
perturbation of a homogeneous random walk. For more references we refer 
the reader to \cite{AvdHoRe2}. Some papers allow for a mutual interaction 
between the walk and the environment. For an example where the jumps of 
the walk depend on the environment in a non-local manner, see \cite{dHodSaSi}. 

In \cite{AvdHoRe1}, a strong law of large numbers was proved for finite-range 
random  walks on a class of interacting particle systems of type (3) that satisfy 
a space-time mixing property called \emph{cone-mixing}. The latter can be 
loosely described as the requirement that the law of the states of the interacting 
particle system inside a space-time cone opening upwards is close to equilibrium 
\emph{conditional} on the states inside a space plane far below the tip. The proof 
was based on a \emph{regeneration-time} argument, showing that there are 
infinitely many space-time points at which the walk stands still for a long time, 
allowing the environment to lose memory. All uniquely ergodic attractive spin-flip 
systems for which the coupling time at the origin has finite mean are cone-mixing. 
However, independent random walks are \emph{not} cone-mixing. Indeed, 
\emph{particle conservation destroys the cone-mixing property}, which is 
why Theorem\ref{thm1} covers new ground. Other examples of dynamic random 
environments that are not cone-mixing for which a strong law of large number for 
the random walk has been proved can be found in \cite{AvdSaVo} (one-dimensional 
exclusion process and $v, v' > 0$ large), \cite{AvFrJaVo} (one-dimensional exclusion 
process speeded up in time) and \cite{dHodSa} (one-dimensional supercritical contact 
process).

%%%%%%%%%%%%%%%%%%%%%%%%%%%%%

\subsection{Open problems and outline}
\label{S1.4}

It remains an open problem to extend Theorem \ref{thm1} to a weak law of 
large numbers for finite $\mu$, i.e., to show that for every $\mu>0$ there 
exists a $v(\mu)\in\R^d$ such that, for every $\ep>0$,
\beq
\label{wLLN}
\lim_{t\to\infty} P^\mu\{\|t^{-1}\kG(t)-v(\mu)\|>\ep\} = 0.
\eeq
The speed in \eqref{wLLN} will be necessarily of the form
\beq
v(\mu) = \rho(\mu)\,v'+[1-\rho(\mu)]\,v'' 
\eeq
for some $\rho(\mu) \in [0,1]$, the latter representing the limiting fraction of time 
the green particle sees a red particle. We should not expect that $\rho(\mu)
=P^\mu\{N(0,0)\geq 1\}=1-e^{-\mu}$. Indeed, since $\rho(\mu)$ is a functional 
of the \emph{environment process}, i.e., the environment as seen relative to 
the location of the walk, we should not expect that $\rho(\mu)$ is a simple 
function of $\mu$.

To appreciate the difficulty of identifying $\rho(\mu)$, note that for 
\emph{static} random environments $\rho(\mu)$ can have anomalous behavior 
as a function of $\mu$. For instance, if we freeze the red particles and
we let the green particle use the transition kernels in \eqref{nnchoice}, 
then it is well-known (see \cite{So}) that
\beq
\label{rhomunnchoice}
\rho(\mu) \left\{\begin{array}{ll}
=\tfrac12, &\quad \text{if } \mu \in [\mu_c^-,\mu_c^+],\\[0.2cm]
>\tfrac12, &\quad \text{if } \mu > \mu_c^+,\\[0.2cm]
<\tfrac12, &\quad \text{if } \mu < \mu_c^-,
\end{array}
\right.
\eeq
with $0 <  \mu_c^- = \log(\tfrac{1}{p}) <  \mu_c^+ = \log(\tfrac{1}{1-p}) < \infty$,
resulting in $v(\mu)=0$ for $\mu \in [\mu_c^-,\mu_c^+]$ and $v(\mu) \neq 0$
elsewhere.

It would be interesting to try and extend Theorem \ref{thm1} (and possibly 
also \eqref{wLLN}) to the case where the dynamic random environment is the 
exclusion process or the zero-range process, both of which fail to be cone-mixing 
as well. These are natural examples that have so far defied a proper analysis.   

The rest of this note is organized as follows. In Section~\ref{S2} we recall
several definitions from \cite{KeSi}. In Section~\ref{S3} we state and prove 
two propositions showing that the green particle is unlikely to visit space-time 
blocks that are not well visited by red particles. In Section~\ref{S4} we use 
these propositions to prove Theorem~\ref{thm1}. In Appendix \ref{appA} we check 
the uniformity in $\mu$ of the estimates in \cite{KeSi}, which is needed in
order to be able to take the limit $\mu\to\infty$.

%%%%%%%%%%% SECTION 2 %%%%%%%%%%%%%%%%%%%%%%%%%%%

\section{Preparations}
\label{S2}

The proof of Theorem~\ref{thm1} will be achieved by showing that the green 
particle spends most of its time in space-time blocks \emph{all} of whose 
points have been visited by a red particle before they are visited by the green 
particle. This will be done separately for ``bad blocks'' and ``good blocks''
(to be defined later) living on successive space-time scales. For the bad 
blocks, most of the work can be lifted from \cite{KeSi}. For the good blocks, 
a percolation-type argument will be used. In the present section we recall 
several definitions from \cite{KeSi}, organized into 4 parts and leading up 
to a key proposition. To simplify notations, we write down the proof for $d=1$ 
and for nearest-neighbor transition kernels $p'$ and $p''$ only. The extension 
to $d \geq 2$ and to finite-range transition kernels will be straightforward.  

\medskip\noindent
{\bf 1.} 
For $t \geq 0$, $\l\in\N_0$, $0\leq s_1<\dots<s_\l\leq t$ and $x_1,\dots,x_\l
\in\Z$, we write $\wh\pi=\wh\pi(\{s_k,x_k\}_{0 \leq k \leq \l})$ for the 
space-time path that, for $1\leq k\leq \l$, jumps to $x_k$ at time $s_k$ and 
stays at $x_k$ during the time interval $[s_k,s_{k+1})$, where we take $s_0=0$,
$x_0=0$ and $s_{\l+1}=t$, i.e., the path takes the value $x_\l$ on $[s_\l,t]$. 
We only consider paths that are contained in the space interval $\kC(t\log t)
=[-t\log t,t\log t]$, and so the class of paths of interest is
\beq
\label{1.1}
\begin{aligned}
\Xi(\l,t) = \big\{\wh\pi=\wh\pi(\{s_k,x_k\}_{0\leq k\leq\l})\colon
&0=s_0<s_1<\dots<s_\l\leq t,\\
&x_k\in\kC(t\log t),\,1 \leq k \leq \l\big\}.
\end{aligned}
\eeq

\medskip\noindent
{\bf 2.}
The renormalization analysis developed in \cite[Section 1 and 4]{KeSi}  depends 
on the choice of a large integer $C_0$ and a strictly increasing sequence of 
positive numbers $(\ga_r)_{r\in\N_0}$ bounded from above by $\tfrac12$ (for 
precise definitions, see (\ref{const1}--\ref{const4}) in Appendix \ref{appA}).  These 
are used to define a sequence of \emph{space-time rectangles} as follows. For 
$r\in\N_0$, abbreviate
\beq
\De_r =C_0^{6r},
\eeq
and, for $i\in\Z$ and $j\in\N$, define (see Fig.~\ref{fig-rect})
\beqa
\label{rblocks}
\qquad\quad \kB_r(i,j) 
&=& [i\De_r,(i+1)\De_r) \times [j\De_r,(j+1)\De_r),\\
\label{largeblocks}
\qquad\quad \wt\kB_r(i,j)
&=& V_r(i) \times [(j-1)\De_r,(j+1)\De_r),\\
\label{pedestal}
\qquad\quad \kV_r(i,j) 
&=& V_r(i) \times \{(j-1)\De_r\},
\eeqa 
with
\beq
\label{vee}
V_r(i) = [(i-3)\De_r,(i+4)\De_r).
\eeq
The $\kB_r(i,j)$'s are called $r$-blocks; $\kV_r(i,j)$ plays the role of the 
pedestal of $\kB_r(i,j)$.

%%%%%%%%%%%%%%% FIGURE %%%%%%%%%%%%%%%%%%%%%%%%%%%%%%%%

\begin{figure}[hbtp]
\vspace{-1.7cm}
\begin{center}
\setlength{\unitlength}{0.5cm}
\begin{picture}(12,10)(0,-2)
{\thicklines 
\qbezier(0,0)(7,0)(14,0) 
\qbezier(0,2)(7,2)(14,2)
\qbezier(0,4)(7,4)(14,4)
\qbezier(0,0)(0,2)(0,4)
\qbezier(14,0)(14,2)(14,4)
\qbezier(6,2)(6,3)(6,4)
\qbezier(8,2)(8,3)(8,4)
}
\put(-.6,-1.5){$i-3$} 
\put(5.8,-1.5){$i$}
\put(7.4,-1.5){$i+1$} 
\put(13.4,-1.5){$i+4$}
\put(-2,0){$j-1$}
\put(-2,2){$j$}
\put(-2,4){$j+1$}
\put(0,0){\circle*{.35}}
\put(14,0){\circle*{.35}}
\put(0,4){\circle*{.35}}
\put(14,4){\circle*{.35}}
\put(6,2){\circle*{.35}}
\put(8,2){\circle*{.35}}
\put(6,4){\circle*{.35}}
\put(8,4){\circle*{.35}}

\end{picture}
\end{center}
\caption{\small Picture of $\kB_r(i,j)$ (small square), $\wt\kB_r(i,j)$ (large rectangle)
and $\kV_r(i,j)$ (base of large rectangle), in units of $\De_r$.}
\label{fig-rect}
\end{figure}

%%%%%%%%%%%%%%%%%%%%%%%%%%%%%%%%%%%%%%%%%%%%%%%%

\medskip\noindent
{\bf 3.} For $r\in\N_0$ and $x\in\Z$, define the space interval
\beq
\kQ_r(x) = [x,x+C_0^r),
\eeq
and, for $t\geq 0$, let
\beqa
\label{badcubes}
U_r(x,t) &=& \sum_{y \in \kQ_r(x)} N(y,t),\\
\label{smallcubes}
E^\mu\{U_r(x,t)\} &=& \mu |\kQ_r(x)| = \mu C_0^r.
\eeqa
We say that $\kB_r(i,j)$ is \emph{bad} if $U_r(x,t)<\ga_r\mu C_0^r$ for some 
$(x,t)$ for which $\kQ_r(x) \times \{t\}$ is contained in $\wt\kB_r(i,j)$, 
i.e., there are significantly fewer red particles than expected in a space 
interval of size $C_0^r = \Delta_r^{1/6}\ll\Delta_r$ somewhere inside the
space-time block $\wt\kB_r(i,j)$. We say that $\kB_r(i,j)$ is \emph{good} 
if it is not bad.

\medskip\noindent
{\bf 4.}
For $r,\l\in\N_0$, define
\beqa
\label{ph11}
\qquad\quad \phi_r(\wh\pi) 
&=& \text{number of bad $r$-blocks that intersect the 
space-time path } \wh\pi,\\
\label{ph22}
\qquad\quad \Phi_r(\l) 
&=&\sup_{\wh\pi\in\Xi(\l,t)} \phi_r(\wh\pi).
\eeqa
The principal result from \cite{KeSi} needed in Section \ref{S3} is the following.

\begin{prop} 
\label{propKeSi}
{\rm (\cite[Proposition 8, p.\ 2441]{KeSi})} For all $K,\ep_0>0$ there exists an 
$r_0=r_0(K,\ep_0)$ such that for all $r\geq r_0$ there exists a $\mu_0 = \mu_0(K,\ep_0,r)$ 
such that for all $r\geq r_0$ and $\mu\geq\mu_0(K,\ep_0,r)$ there exists a 
$t_0=t_0(K,\ep_0,r,\mu)$ such that
\beq
P^\mu\big\{\Phi_r(\l) \geq \ep_0C_0^{-6r}(t+\l)\big\} \leq 2t^{-K},
\qquad r \geq r_0,\,\mu\geq\mu_0,\,t\geq t_0,\,l\in\N_0.
\label{est14}
\eeq
\end{prop}

\noindent In Appendix~\ref{appA} we check the uniformity in $\mu$ of the various estimates that 
went into the proof of Proposition~\ref{propKeSi}.

%%%%%%%%%%%% SECTION 3 %%%%%%%%%%%%%%%%%%%%%%%%%%%%%%%%

\section{Two propositions}
\label{S3}

The proof of Theorem~\ref{thm1} in Section~\ref{S4} will be built on two
propositions, which are stated and proved in Sections~\ref{S3.1} and \ref{S3.4}, 
respectively. The first proposition controls the number of bad $r$-blocks 
$\kB_r(i,j)$ that intersect the path of the green particle up to time $t$, 
and its proof makes use of Proposition~\ref{propKeSi}. The second proposition 
controls the number of good $r$-blocks $\kB_r(i,j)$ that intersect the path 
of the green particle up to time $t$ and contain some point $(x,t)$ that has 
no red particle coming from $\kV_r(i,j)$. The proof of the second proposition 
requires two auxiliary lemmas, which are stated and proved in Sections~\ref{S3.2} 
and \ref{S3.3}, respectively.

%%%%%%%%%%%%%%%%%%%%

\subsection{First proposition}
\label{S3.1}

For $t \geq 0$, let $\kE_1(t)$ denote the event that the number of jumps by the 
green particle up to time $t$ exceeds $2t$. Then there exist $C_1,C_2>0$ such 
that
\beq
\label{numjumps}
P\{\kE_1(t)\} \leq C_1e^{-C_2t}.
\eeq
Indeed, the green particle has constant jump rate 1. Therefore the number of jumps 
up to time $t$ is a Poisson random variable with mean $t$, and the inequality is a 
standard large deviation bound for the Poisson distribution.

Fix $K,\ep_0>0$ and $r_0=r_0(K,\ep_0)$ as in Proposition \ref{propKeSi}. For
$t \geq 0$, let
\beq
\kH(t) = \{\kG(s)\colon\,0 \leq s \leq t\} 
 \text{  be the path of the green particle up to time } t,
\eeq
and
\beq
\label{gamma}
\Ga_r(t) = \big\{(i,j)\colon\, \kB_r(i,j) \cap \kH(t) \neq \emptyset \big\}.
\eeq
The union of the $r$-blocks $\kB_r(i,j)$ with $(i,j) \in \Ga_r(t)$ is a fattened-up 
version of the path of the green particle. We want to prove that, at large times
$t$ and high densities $\mu$, the green particle sees many red particles close by. 
In fact, we will prove a somewhat stronger statement, namely, that with a large 
probability in an $r$-block $\kB_r(i,j)$ that is visited by the green particle all 
the space-time points are visited by a red particle coming from $\kV_r(i,j)$.

For $t\geq 0$ and $r\in\N_0$, let
\beq
\wt\Ga_r(t)= \big\{(i,j)\colon\,\kB_r(i,j) \cap \kH(t) \neq \emptyset,\,
\kB_r(i,j) \text{ is bad}\big\},
\eeq 

\begin{prop}
\label{prop1}
For $K,\ep_0>0$, $r \geq r_0, \mu \geq \mu_0$ and $t$ sufficiently large,
\beq
\label{badpairs}
P^\mu\big\{|\wt\Ga_r(t)| \geq 3\ep_0 C_0^{-6r}t\big\}
= P\big\{\phi_r(\kH(t)) \geq 3\ep_0 C_0^{-6r} t\big\}
\leq 3t^{-K}.
\eeq
\end{prop}

\begin{proof} 
The equality follows from \eqref{ph11}. To obtain the inequality, we apply 
Proposition \ref{propKeSi}. This tells us that for $r \geq r_0$, $\mu\geq\mu_0$,
$t\geq t_0$ and $\l\in\N_0$, outside an event $\kE_2(t)$ of probability at most 
$2t^{-K}$, we have
\beq
\Phi_{r}(\l) = \sup_{\wh\pi\in\Xi(\l,t)} \phi_{r}(\wh\pi) \leq
\ep_0 C_0^{-6r} (t+\l).
\eeq
Furthermore, since each jump has size 1, if $\kH(t)$ makes exactly $\l$ jumps 
with $0\leq\l\leq\kC(t\log t)$, then $\kH(t)\in\Xi(\l,t)$. Hence, for $\log t 
\geq 2$ and outside the event $\kE_1(t)\cup\kE_2(t)$, we have
\beq
\label{phirest}
\phi_r(\kH(t)) \leq \Phi_r(\l) \leq \ep_0 C_0^{-6r} (t+\l) \leq
3 \ep_0 C_0^{-6r}t.
\eeq
Combine \eqref{numjumps} and \eqref{phirest}, and choose $t$ so large that
$C_1e^{-C_2t} \leq t^{-K}$, to get \eqref{badpairs}.
\end{proof}

%%%%%%%%%%%%%%%%%%%%%%%%%%%%%%%%%%

\subsection{First auxiliary lemma}
\label{S3.2}

The time coordinate of the green particle is just time itself. Hence, if $T$ is some 
space-time set with projection $\wt T$ onto the time-axis, then the total time spent 
by the green particle inside $T$ is at most the Lebesgue measure of $\wt T$. In 
particular, if $\kH(t)$ intersects no more than $3\ep_0 C_0^{-6r}t$ bad $r$-blocks, 
then the total time that is spent by the green particle in bad $r$-blocks up to time
$t$ is at most
\beq
\begin{split}
3\ep_0 C_0^{-6r}t \times C_0^{6r} = 3\ep_0 t.
\label{est3}
\end{split}
\eeq
We want to control the set of space-time points $(x,t)$ in a good $r$-block 
$\kB_r(i,j)$ that intersects $\kH(t)$ such that there is no red particle at 
$(x,t)$ coming from $\kV_r(i,j)$. We want to show that also this set is small 
with a large probability. 

Let
\beq
\kF(t) = \sigma\{N(s)\colon\,0\leq s\leq t\},
\eeq
be the sigma-field generated by the paths of all the red particles up to time $t$,
and define
\beq
\begin{aligned}
\kE_r(i,j) &= \big\{\exists\,(x,t)\in\kB_r(i,j)\colon\\
&\qquad \text{no red particle coming from } \kV_r(i,j) \text{ hits } (x,t)\big\}.
\end{aligned}
\eeq

\begin{lem}
\label{lem2}
For all $\ep_1>0$ and $r\in\N_0$ there exists a $\mu_1=\mu_1(\ep_1,r)$ such that 
for all $\mu\geq\mu_1$, $i\in\Z$ and $j\in\N$, and uniformly on the event
\beq
\label{est5}
\kN_r(i,j) = \left\{\sum_{x \in V_r(i)} N(x,(j-1)\De_r) 
\geq \ga_0\mu\De_r\right\},
\eeq
the following holds:
\beq
P^\mu \big\{\kE_r(i,j) \mid \kF((j-1)\De_r)\} \leq \ep_1.
\label{T21}
\eeq
\end{lem}

\begin{proof} 
Note that $\kE_r(i,j)$ depends only on the paths during the time interval 
$[(j-1)\De_r,(j+1)\De_r]$ of the red particles located in the space interval 
$V_r(i)=[(i-3)\De_r,(i+4)\De_r]$ at time $(j-1)\De_r$. Since the red particles 
are interchangeable, the conditional probability in \eqref{T21} in fact only 
depends on $N(x,(j-1)\De_r)$, $x\in V_r(i)$.

It is easy to see that if there are at least $8\De_r$ particles in $V_r(i)$ 
at time $(j-1)\De_r$, then the conditional probability $f(r)$ that these 
particles after time $(j-1)\De_r$ move in such a way that there is at least 
one of them at each point $(x,t) \in \kB_r(i,k)$ satisfies $f(r)>0$ (see 
Fig.~\ref{fig-rect}). In particular, $f(r)$ can be taken to be independent 
of the location of the red particles at time $(j-1)\De_r$. In other words, on 
the event
\beq
\sum_{x \in \kV_r(i)}  N(x,(j-1)\De_r) \geq 8\De_r,
\label{function}
\eeq
we have
\beq
P^\mu\big\{[\kE_r(i,j)]^c \mid \kF((j-1)\De_r)\} \geq f(r).
\label{est6}
\eeq

Assume now that \eqref{est5} holds. Then there are at least $\ga_0\mu\De_r$ red 
particles in $V_r(i)$ at time $(j-1)\De_r$. Order these particles in an 
arbitrary way and partition them into $q= \lfloor\ga_0\mu/8\rfloor$ subsets 
of at least $8\De_r$ particles each, ignoring what is left over. For each of these 
subsets the bound in \eqref{est6} is valid. The event in the left-hand side of 
\eqref{T21} occurs if and only if the event in the left-hand side of \eqref{est6} 
fails for each of the $q$ subsets. Since the paths of disjoint sets of red particles 
are independent, the left-hand side of \eqref{T21} is therefore at most $[1-f(r)]^q$.
Now take $\mu$ so large that $[1-f(r)]^{(\gamma_0\mu/8)-1} \leq \ep_1$, i.e.,
\beq
\mu \geq \mu_1 =\frac{8}{\ga_0}\left[1 + \frac{\log\ep_1}{\log[1-f(r)]}\right].
\eeq
Then \eqref{T21} follows.
\end{proof}

Note that because $r \mapsto \gamma_r$ is non-decreasing (see Appendix \ref{appA})
the result in Lemma {\rm \ref{lem2}} also holds when $\gamma_0$ is replaced by 
$\gamma_r$ in \eqref{est5}. In what follows we will use the version with $\gamma_0$.

%%%%%%%%%%%%%%%%%%%%%%%

\subsection{Second auxiliary lemma}
\label{S3.3}

Abbreviate
\beq
\label{est8}
M = \ga_0\mu\De_r
\eeq
and, for $s\in\N$, $i_1,\dots,i_s\in\Z$ and $j_1,\dots,j_s\in\N$, introduce the 
events
\beq
\label{8D}
\kD_r(i_1,j_1,\dots,i_s,j_s) = \bigcap_{u=1}^s \kN_r(i_u,j_u).
\eeq

\begin{lem}
\label{lem3}
For all $\ep_1>0$ and $r\in\N_0$ there exists a $\mu_1=\mu_1(\ep_1,r)$ such that 
for all $\mu\geq\mu_1$ the following are true.\\
(a)
Let $i_1,\dots, i_s$ be such that their mutual differences are all $\geq 8$. Then,
for all $j\in\N$,
\beq
\label{est9b}
\begin{split}
&P^\mu\Big\{\kD_r(i_1,j,\dots,i_s,j) 
\cap \big[\cap_{u=1}^s \kE_r(i_u,j)\big]\Big\}\\
&\qquad \leq P^\mu\{\kD_r(i_1,j,\dots,i_s,j)\}\,\ep_1^s
\leq \ep_1^s, \qquad \mu\geq\mu_1(\ep_1,r).
\end{split}
\eeq
(b) Let $(i_1,j_1),\dots,(i_s,j_s)$ be distinct such that the sum of their 
componentwise mutual differences are all $\geq 10$ and $j_1,\dots,j_s$ all
have the same parity. Then
\beq
\label{est9}
P^\mu\Big\{\kD_r(i_1,j_1,\dots,i_s,j_s)
\cap \big[\cap_{u=1}^s \kE_r(i_u,j_u)\big]\Big\}
\leq \ep_1^s, \qquad \mu \geq \mu_1(\ep_1,r).
\eeq
\end{lem}

\begin{proof}
(a) Write the left-hand side of \eqref{est9b} as a conditional expectation given 
$\kF((j-1)\De_r)$. Since the $j_u$'s coincide, $\kD(i_1,j,\dots,i_s,j)$ only 
depends on the sites in $\Z \times [0,(j-1)\De_r]$. On the other hand, $\kE_r(i,j)$ 
only depends on the red particles at the points in $\Z\times [(j-1)\De_r,(j+1)\De_r)$. 
As in the first part of the proof of Lemma~\ref{lem2}, the conditional distribution 
of $\cap_{u=1}^s \kE_r(i_u,j)$ given $\kF((j-1)\De_r)$ only depends on $N(x,(j-1)
\De_r)$, $x\in\Z$. In fact, the conditional distribution of $\kE_r(i_u,j)$ given 
$\kF((j-1)\De_r)$ only depends on $N(x,(j-1)\De_r)$, $x \in V_r(i_u)$. The collections 
of red particles counted by $N(x,(k-1)\De_r)$, $x\in V_r(i_u)$, for different $i_u$ 
are disjoint, because the intervals $V_r(i_u)$ for different $i_u$ are disjoint. 
It follows that $\kE_r(i_u,j)$, $1\leq u\leq s$, are conditionally independent 
given $\kF((j-1)\De_r)$. Therefore the left-hand side of \eqref{est9b} is bounded 
from above by
\beq
\begin{aligned}
&E^\mu\Big\{
P^\mu\big\{\kD_r(i_1,j,\dots,i_s,j) \mid \kF((j-1)\De_r)\big\}\\
&\qquad \times \prod_{u=1}^s
P^\mu\big\{\kE_r(i_u,j) \mid \kF((j-1)\De_r)\big\}\Big\}.
\end{aligned}
\eeq
However, on the event $\kD(i_1,j,\dots,i_s,j)$, \eqref{est5} holds for 
$i\in\{i_1,\dots,i_s\}$, and so we see from Lemma~\ref{lem2} that
\beq
P^\mu\{\kE_r(i_u,j)\mid \kF((j-1)\De_r)\} \leq \ep_1,
\eeq
provided we take $\mu\geq\mu_1 = \mu_1(\ep_1,r)$. This gives the first inequality
in \eqref{est9b}. The second inequality is trivial.

\medskip\noindent
(b) Put
\beq
j = \max\{j_1,\dots,j_s\},
\eeq
and suppose, without loss of generality, that there exists a $1 \leq \bar{u} 
\leq s$ such that
\beq
j_u < j \text{ for } u \leq \bar{u} \text{ and } j_u = j \text{ for } j > \bar{u}. 
\eeq
Note that
\beq
\begin{split}
&\kD_r(i_1,j_1,\dots,i_s,j_s)\\
&\qquad = \kD_r(i_1,j_1,\dots,i_{\bar{u}},j_{\bar{u}}) 
\cap \kD_r(i_{\bar{u}+1},j_{\bar{u}+1},\dots,i_s,j_s).
\end{split}
\eeq
Note further that $j_1,\dots,j_{\bar{u}} \leq k-2$, because all $j_u$ have the 
same parity. This implies that $\kD_r(i_1,k_1,\dots,i_{\bar{u}},j_{\bar{u}})$ 
is $\kF((j-1)\De_r)$-measurable. Consequently, as in the proof of part (a), on 
the event $\kD_r(i_1,j_1,\dots,i_{\bar{u}},j_{\bar{u}})$ we have
\beq
\begin{split}
&P^\mu\big\{\cap_{u=\bar{u}+1}^s \kE_r(i_u,j_u) \mid \kF((j-1)\De_r)\big\}\\[0.2cm]
&\qquad =P^\mu\big\{\cap_{u=\bar{u}+1}^s \kE_r(i_u,j) \mid \kF((j-1)\De_r)\big\}\\
&\qquad = \prod_{u=\bar{u}+1}^s P^\mu\big\{\kE_r(i_u,j) \mid \kF((j-1)\De_r)\big\}
\leq \ep_1^{s-\bar{u}}.
\end{split}
\eeq
By taking the conditional expectation with respect to $\kF((j-1)\De_r)$ and using 
part (a), we obtain
\beq
\label{est10}
\begin{split}
&P^\mu\big\{\kD_r(i_1,j_1,\dots,i_s,j_s) 
\cap \big[\cap_{u=1}^s \kE_r(i_u,j_u)\big]\big\}\\
&\leq E^\mu\Big\{\kD_r(i_1,j_1,\dots,i_u,j_u) 
\cap \big[\cap_{u=1}^{\bar{u}}
\kE_r(i_u,j_u)\big]~\big|~\kF((j-1)\De_r)\Big\}\,\ep_1^{s-\bar{u}}\\
&= P^\mu \Big\{\kD_r(i_1,j_1,\dots,i_{\bar{u}},j_{\bar{u}}) 
\cap \big[\cap_{u=1}^{\bar{u}} \kE_r(i_u,j_u)\big]\Big\}\,\ep_1^{s-\bar{u}}.
\end{split}
\eeq
The proof can now be completed via a recursive argument. Indeed, the left-hand side 
of \eqref{est10} deals with the probability of events indexed by $s$ pairs $(i_u,j_u)$ 
with $j_u \leq j$ and estimates this probability in terms of probabilities of events 
indexed by pairs $(i_u,j_u)$ with $j_u \leq j-1$ (and powers of $\ep_1$). We can 
therefore iterate the estimate until it only contains powers of $\ep_1$.
\end{proof}

%%%%%%%%%%%%%%%%%%%

\subsection{Second proposition}
\label{S3.4}

Associated with $\kH(t)$ is the collection of pairs $\Ga_r(t)$ introduced in 
\eqref{gamma}. Because the jumps of the green particle have size 1, $\Ga_r(t)$ 
when viewed as a subset of $\Z^2$ is connected, i.e., for any two pairs 
$(i',j'),(i'',j'')\in\Ga_r(t)$ there is a path in $\Ga_r(t)$ that runs from 
$(i',j')$ to $(i'',j'')$. In other words, $\Ga_r(t)$ is a so-called \emph{lattice 
animal} containing the origin. We claim that with probability at least $1-C_1
e^{-C_2t}$ this lattice animal contains at most $\l+t \leq 3t$ sites. Indeed, 
this is so because $\kH$ can go from an $r$-block to an adjacent $r$-block in 
only two ways:
\begin{itemize} 
\item[(i)] 
It crosses one of the time lines $j\De_r$, $j\in\N$, without making a jump. 
Since the time between two successive such crossings is $\De_r$, at most 
$t/\De_r$ such crossings can occur up to time $t$. 
\item[(ii)] 
It makes a jump. By the definition of $\Xi(l,t)$, for $\Ga(t)\in\Xi(\l,t)$ 
there are exactly $\l$ such jumps up to time $t$.
\end{itemize}
Now, it is well known that there exist constants $C_3,C_4$ such that the number of 
lattice animals of size $3t$ containing the origin is bounded from above by 
$C_3e^{C_4t}$. Thus, if we define
\beq
\kW_r(t) = \text{ collection of possible sets } \Ga_r(t),
\eeq
then we have proved that, with probability at least $1-C_1e^{-C_2t}$,
\beq
|\kW_r(t)| \leq C_3e^{C_4t}.
\label{est13}
\eeq

\medskip\noindent
The $r$-block corresponding to a point $(i,j)\in\Ga_r(t)$ can be either bad or good. 
We will call a pair $(i,j)\in\Ga_r(t)$ bad or good according as $\kB_r(i,j)$ is bad 
or good. It is immediate from (\ref{ph11}--\ref{est14}) that, outside the event 
$\kE_1(t)\cup\kE_2(t)$, the number of bad $r$-blocks in $\Ga_r(t)$ is at most 
$3\ep_0 C_0^{-6r}t$. Together with \eqref{est3}, this proves that $\kH(t)$ spends 
only a small fraction of its time in bad blocks. We therefore only need to deal 
with the subset of good pairs in $\Ga_r(t)$. Of particular interest will be the 
following subset of $\Ga_r(t)$:
\beq
\La_r(t) = \big\{(i,j)\colon\, \kB_r(i,j) \cap \kH(t) \neq \emptyset,\,
\kB_r(i,j) \text{ is good}, \, \kE_r^*(i,j) \text{ occurs}\big\},
\eeq
where $\kE_r^*(i,j)$ is the event that $\kB_r(i,j)$ contains a point that is not 
visited by any of the red particles coming from $\kV_r(i,j)$.

\begin{prop}
\label{prop4}
There exist a $C_1,\dots,C_6>0$ such that for all $0<\ep_1<1$, $\mu\geq\mu_1(\ep_1,r)$ 
and $t$ sufficiently large,
\beq
\label{est11}
P^\mu\{|\La_r(t)| \geq \ep_1 t\} \leq C_1e^{-C_2t} 
+ C_3e^{C_4t}2^{3t}\ep_1^{\lfloor C_5 + C_6\ep_1t\rfloor}.
\eeq
\end{prop}

\begin{proof}
The idea of the proof is to partition the event $\{|\La_r(t)|\geq\ep_1t\}$ into 
a union of subevents of the form analyzed in Lemma~\ref{lem3}, to estimate the 
probability of each of these subevents by means of Lemma~\ref{lem3}, and afterwards 
sum over all the ways to do the partition.

Consider a sample point for which $|\La_r(t)| \geq \ep_1 t$, but for which $\kE_1(t)$ 
does not occur (i.e., the green particle makes $\leq 2t$ jumps up to time $t$). 
Then, since $\La_r(t)\subset\Ga_r(t)\subset\Z^2$, there exist at most $\lfloor
C_5+C_6\ep_1 t\rfloor$ points $(i_u,j_u)\in\La_r(t)$ with the sum of their 
componentwise differences all $\geq 10$, with $C_5,C_6>0$ independent of 
$\La_r(t)$ and $\mu$. Since $\kB_r(i,j)$ is good when $(i,j)\in\La_r(t)$, we have
\beq
\sum_{x \in V_r(i)} N(x,(j-1)\De_r) \geq \ga_r \mu \De_r \geq \ga_0\mu\De_r = M,
\label{est25}
\eeq
i.e., the event $\kN_r(i,j)$ defined in \eqref{est5} occurs. Let
\beq
\wh\La_r(t) = \big\{(i,j)\colon\, 
\kB_r(i,j) \cap \kH(t) \neq \emptyset,\,\kN_r(i,j) \text{ and } \kE_r^*(i,j)
\text{ occur}\big\}.
\eeq
Then $\wh\La_r(t)\supset\La_r(t)$, and so the points $(i_u,j_u)$, $1\leq u\leq 
C_5+C_6\ep_1t$, all lie in $\wh\La_r(t)$. This means that $\kD_r(i_1,j_1,\dots,
i_s,j_s)$ occurs with $s = \lfloor C_5 + C_6\ep_1t\rfloor$ and $M$ as in 
\eqref{est8} (recall \eqref{8D}). In addition,
\beq
\cap_{u=1}^s \kE_r(i_u,j_u)
\eeq
occurs (recall \eqref{T21}).

The above observations show that $|\La_r(t)|>\ep_1t$ can occur for $\mu\geq
\mu_1(\ep_1,r)$ and $t$ large enough only if, for some possible choice of 
the $(i_u,j_u)$'s,
\beq
\kD(i_1,j_1,\dots,i_s,j_s) \cap \big[\cap_{u=1}^s \kE_r(i_u,j_u)\big]
\label{est8DK}
\eeq
occurs. Lemma~\ref{lem3} shows that, for any permissible choice of the $(i_u,j_u)$'s, 
the probability of \eqref{est8DK} is at most $\ep_1^s$. Consequently,
\beq
\begin{split}
&P^\mu\{|\La_r(t)| \geq \ep_1 t\}\\
&\qquad \leq \sum_{\Ga_r} \sum_{\Th \subset \Ga_r} 
P^\mu\big\{\kD_r(i_1,j_1,\dots,i_s,j_s) 
\cap \big[\cap_{u=1}^s \kE_r(i_u,j_u)\big]\big\}\\
&\qquad \leq \sum_{\Ga_r} \sum_{\Th \subset \Ga_r} \ep_1^s,
\label{est26*}
\end{split}
\eeq
provided $\mu \geq \mu_1(\ep_1,r)$ and $t$ is sufficiently large. Here, the sum 
over $\Ga_r$ runs over all possible realizations of the random set $\Ga_r(t)$, 
and $\Th$ runs over all choices of the $(i_u,j_u)$'s whose sum of componentswise 
differences are all $\geq 10$, and are such that the $j_u$'s all have the same 
parity. 

Now assume that our sample point lies outside $\kE_1(t)$, which happens with a 
probability at least $1-C_1e^{-C_2t}$. Then
\beq
\sum_{\Th \subset \Ga_r} 1 \leq 2^{3t},
\eeq
because $|\Ga_r|\leq 3t$, as we saw before. Moreover, by \eqref{est13}, we have
\beq
\sum_{\Ga_r} 1 \leq C_3e^{C_4t},
\eeq
since $\Ga_r$ is a lattice animal that contains the origin and has size at most $3t$. 
Combining these estimates, we find that for $\mu\geq\mu_1(\ep_1,r)$ and $t$ sufficiently
large,
\beq
P^\mu\{|\La_r(t)| \geq \ep_1 t\} 
\leq C_1e^{-C_2t} + C_3e^{C_4t}2^{3t}\ep_1^{\lfloor C_5 + C_6\ep_1t\rfloor}.
\label{estHK}
\eeq
\end{proof}

%%%%%%%%%% SECTION 4 %%%%%%%%%%%%%%%%%%%%%%%%%%%%%%%%%%

\section{Proof of Theorem \ref{thm1}}
\label{S4}

\begin{proof}
With Propositions~\ref{prop1} and \ref{prop4} in hand, the proof of Theorem~\ref{thm1} 
is routine. We distinguish four possible types of $r$-blocks, 
\beq
\mathrm{(bad,occupied)},\quad \mathrm{(bad,vacant)},
\quad \mathrm{(good,occupied)}, \quad \mathrm{(good,vacant)}, 
\eeq
where an $r$-block $\kB_r(i,j)$ is called \emph{occupied} when $\kE_r^*(i,j)^c$ 
occurs, i.e., every point in $\kB_r(i,j)$ is visited by a red particle coming
from $\kV_r(i,j)$, and is called \emph{vacant} otherwise. The number of $r$-blocks 
of type (bad,occupied) that intersect $\kH(t)$ will be denoted by $N_r(t;
\mathrm{(bad,occupied)})$, and similarly for the other types.

We have shown in Proposition~\ref{prop1} that
\beq
\label{Nest1}
\begin{split}
N_r(t;\mathrm{(bad)}) 
&= N_r(t;\mathrm{(bad,occupied)}) + N_r(t;\mathrm{(bad,vacant)})\\
&\leq |\wt\Ga_r(t)| \leq 3\ep_0 C_0^{-6} t
\end{split}
\eeq
outside an event of probability at most $3t^{-K}$, provided $K,\ep_0>0$, $\mu\geq
\mu_0(K,\ep_0,r)$ and $t$ is sufficiently large. We have further shown in 
Proposition~\ref{prop4} that
\beq
\label{Nest2}
N_r(t;\mathrm{(good,vacant)}) = |\La_r(t)| \leq \ep_1 t
\eeq
outside an event of probability at most $C_1e^{-C_2t}+C_3e^{C_4t}2^{3t}
\ep_1^{\lfloor C_5 + C_6\ep_1t\rfloor}$, provided $0<\ep_1<1$, $\mu\geq
\mu_1(\ep_1,r)$ and $t$ is sufficiently large. Since we can choose 
$\ep_0,\ep_1$ arbitrarily small, it follows from (\ref{Nest1}--\ref{Nest2}) 
that there exists a function $\mu\mapsto\wh{t}(t)$ from $(0,\infty)$ to 
itself such 
that
\beq
\begin{split}
&\frac1t\,\big[N(r,t;\mathrm{(bad,vacant)})
+ N(r,t;\mathrm{(good,vacant)})\big] \to 0\\
&\text{in } P^\mu\text{-probability as } \mu,t\to\infty 
\text{ such that } t\geq\wh t(\mu).
\end{split}
\eeq
According to the projection argument given in the lines just before \eqref{est3}, 
this in turn implies that
\beq
\begin{split}
&\frac1t\,\big[\text{total time that $\kH(t)$ is in an $r$-block}\\
&\quad \text{ that is (bad,occupied) or (good,occupied)\big]} \to 1\\
&\text{in } P^\mu\text{-probability as } \mu,t\to\infty 
\text{ such that } t\geq\wh t(\mu).
\end{split}
\label{problim}
\eeq

Finally, let $L$ be the infinitesimal generator of the random walk performed 
by the green particle when the space-time trajectories of the red particles,
given by $N$ defined in \eqref{Ndef}, are fixed. Then
\beq
(LI)(x,t) = \begin{cases} 
v'  &\text{ if } N(x,t) \geq 1,\\
v'' &\text{ if } N(x,t) = 0,
\end{cases}
\qquad I(x,t)=(x,t).
\label{generator}
\eeq
Recall $\kG$ defined in \eqref{Gdef}. By a standard martingale property, we 
have (see e.g.\ \cite{KeSi}(Lemma 10))
\beq
\label{mart}
\lim_{t\to\infty} \frac1t\,\left[\kG(t)-\int_0^t (LI)(\kG(s),s)\,ds\right] = 0, 
\qquad P^\mu\text{-a.s.}
\eeq
Combining \eqref{problim} and \eqref{generator}, we see that $(LI)(\kG(s),s) = v'$ 
on a set of $s$-values in $[0,t]$ that converges in $P^\mu$-probability to all of 
$[0,t]$. Therefore Theorem~\ref{thm1} follows from \eqref{mart}.
\end{proof}

%%%%%%%%%%%%%% Appendix %%%%%%%%%%%%%%%%%%%%%%%%%%%%%%%%%

\appendix

\section{Uniformity in $\mu$}
\label{appA}

Our main focus in this Appendix is on \cite[Section 4]{KeSi}, and we will adopt the 
notation used there. Consequently, the arguments given below cannot be read 
independently. Moreover, in \cite{KeSi} at some places a choice of parameters 
depends on $\mu$. However, as we will check below, this dependence does not 
require $\mu\to\infty$. 

Note that in Proposition~\ref{propKeSi} we choose the parameters in the order 
$K,\ep_0,r,\mu,t$. Since in Theorem~\ref{thm1} we let $t\to\infty$ first, it 
suffices to check that inequalities hold for large $t$ when $\mu$ is fixed. 

\medskip\noindent
{\bf 1.} As shown in \cite[Sections 1 and 4]{KeSi}, $C_0$ and $(\gamma_r)_{r\in\N_0}$ 
mentioned in Section~\ref{S2}, Part 2, are chosen such that
\beqa
\label{const1}
&0 <\ga_0 \prod_{j=1}^\infty \big[1-2^{-j/4}\big]^{-1} \leq \frac12,\\
\label{const2}
&\ga_1=\ga_0,\qquad \ga_{r+1} 
= \ga_0\prod_{j=1}^r\big[1-C_0^{-j/4}\big]^{-1}, \quad r\in\N,
\eeqa
where $C_0$ is taken so large that, for all $r\in\N$,
\beqa
\label{const3}
&C_0^{-r/2}-\Big(1-\frac{C_4r\log C_0}{C_0^r}\Big)
\big(1-e^{-C_0^{-r/2}}\big)[1-C_0^{-r/4}]^{-1}\le -\frac12C_0^{-3r/4},\\
\label{const4}
&9C_0^{12(r+1)}\exp[-\frac 12 \ga_0 \mu C_0^{r/4}]\leq 1,
\eeqa
where $C_4$ is the constant in \cite[Lemma 5]{KeSi} (and $\mu$ takes over the role
of $\mu_A$ in \cite{KeSi}). It is not hard to check in the proof of \cite[Lemma 5]{KeSi} 
that $C_4$ in \eqref{const3} can be chosen independently of $\mu$. As was already 
checked in \cite{KeSi}, $C_4$ is also independent of $C_0$, so that we can choose 
$(\ga_r)_{r\in\N_0}$ and $C_0$ independently of $C_4$ as well. In other words, once 
a value for $C_4$ has been determined on the basis of \cite[Lemma 5]{KeSi}, we may 
safely let $\mu\to\infty$.

\medskip\noindent
{\bf 2.} The next place were the uniformity in $\mu$ needs to be checked is 
\cite[Eq.\ (4.16--4.17)]{KeSi}. For any choice of $K>0$ and $K_4>0$, we define 
$R(t)$ by \cite[Eq.\ (4.16)]{KeSi}, the position at time $t$ of the right-most red 
particle. The first three inequalities in \cite[Eq.\ (4.17)]{KeSi} continue to be 
valid, uniformly in $\mu$. We can choose $K_4$ to make also the fourth 
inequality in \cite[Eq.\ (4.17)]{KeSi}valid, uniformly in $\mu$, by observing that 
$U_r(x,t)$ in \eqref{badcubes} has a Poisson distribution with mean $\mu C_0^r$, 
so that, for any $\th>0$,
\beq
\begin{split}
P^\mu\{U_r(x,t) \leq \tfrac12\mu C_0^r\}
&= P^\mu\left\{e^{-\theta U_r(x,t)} \geq e^{-\tfrac12\th\mu C_0^r}\right\}\\
&\leq e^{\tfrac12\th\mu C_0^r}\,E^\mu\left\{e^{-\theta U_r(x,t)}\right\}
=e^{\tfrac12\th\mu C_0^r}\,e^{-\mu C_0^r(1-e^{-\th})}.
\end{split}
\eeq
Pick $\th>0$ so small that $(1-e^{-\th})-\tfrac12\th \geq \tfrac14\th$, 
to obtain
\beq
\begin{split}
\sum_{r \geq R(t)} P\left\{U_r(x,t) \leq \tfrac12 \mu C_0^r\right\}
&\leq e^{-\tfrac14\th\mu C_0^{R(t)}}\,
\sum_{r \geq R(t)} e^{-\tfrac14\th\mu C_0^{r-R(t)}}\\
&\leq K_{20} e^{-\tfrac14\th\mu C_0^{R(t)}}
\end{split}
\eeq
for some constant $K_{20}$ that depends on $C_0$ only, as long as $\mu$ is bounded
away from 0, say $\mu \geq 2$. It is immediate from this estimate that the sum of 
$P\{U_r(x,t)\leq\tfrac12\mu C_0^r\}$ over $(x,s)$, with $s \in [-\De_r,t+\De_r)$ integer 
and $x$ such that $\kQ_r(x)$ intersects $\kC(t\log t + 3\De_r)$, when summed over 
$r\geq R(t)$ is no more than $t^{-K}$, provided $\mu \geq 2$ and $K_4$ (or, equivalently, 
$R(t)$) is sufficiently large, independently of $\mu$. Thus, we conclude that
\cite[Eq.\ (4.17)]{KeSi} is valid uniformly in $\mu\geq 2$.

\medskip\noindent
{\bf 3.} \cite[Lemmas 5--6]{KeSi} remain valid for $\mu>0$, while also
\cite[Lemma 7]{KeSi} remains valid, even with $C_5$ independent of $\mu$, 
as long as $\mu$ is bounded away from 0, say $\mu\geq 2$. Indeed, 
the inequality in \cite[Eq.\ (4.29)]{KeSi} is based on the estimate
\beq
E\{T\} = \la \nu^2\rho_{r+1} \leq 2^2 \la \leq 6^2
\frac{t+\l}{\nu \De_{r+1}}
\eeq
and on Bernstein's inequality (see \cite[Eq.\ (4.37) and subsequent lines]{KeSi} 
for the appropriate notation. In the case of a binomial random variable 
$T$ corresponding to $\la \nu^2$ trials with success probability $\rho$, 
Bernstein's inequality gives
\beq
P\{T \geq \al E\{T\}\} = P\{T - E\{T\} \geq (\al-1)E\{T\}\}
\le \exp\big[-\tfrac14(\al-1)E\{T\}\big]
\eeq
for $\al>1$ (see \cite[Exercise 4.3.14]{ChTe}. Thus, \cite[Lemma 7]{KeSi} 
holds uniformly in $\mu\geq 2$.

\medskip\noindent
{\bf 4.} It remains to verify the uniformity of \cite[Proposition 8]{KeSi}, 
i.e., Proposition~\ref{propKeSi} above. To do so, we consider a sample point 
where \cite[Eq.\ (4.39)]{KeSi} holds for all $\mu\geq 2$, $r \geq R(t)$ and $\l\in\N_0$, 
and where \cite[Eq. (4.40)]{KeSi} holds independently of $\mu$. Then there exists an 
$r_0$ independent of $\mu\geq 2$ such that \cite[Eq.\ (4.41)]{KeSi} holds all $r \in 
[r_0,R(t)-1]$, $\l \in\N_0$, $\mu\geq 2$ and $t$ sufficiently large. Moreover, 
these estimates hold for all sample points outside an event of probability at most
\beq
\label{last}
t^{-K} + \sum_{r=1}^{R(t)-1} \sum_{\l \in \N_0} 
\exp\Big[-(t+\l)C_5\kappa_0\exp\Big[\tfrac12\ga_r\mu C_0^{\tfrac14}\Big]\Big] 
\leq 2t^{-K}.
\eeq
This proves Proposition~\ref{propKeSi} with the required uniformity in $\mu$.
The only requirement for $r_0$ is that the last inequality in \cite[Eq.\ (4.41)]{KeSi} 
holds. Thus, we only need
\beq
6\kappa_0[C_0]^{12+6r} \exp\Big[-\frac{\ga_0\mu}{4}C_0^{\tfrac14 r}\Big] 
\leq \ep_0.
\eeq

\bigskip

\noindent {\bf{Acknowledgements.}} 
Part of the research for this paper was carried out in the Spring of 2009, during 
visits by the authors to the Mittag-Leffler Institute in Djursholm in Sweden and
to the Mathematical Institute of Leiden University in The Netherlands. The authors
thank the Swedish Research Council for support. FdH is supported by ERC Advanced 
Grant 67356 VARIS. VS is supported by CNPq (Brazil) grants 308787/11-0 and 
484801/11-0 and by Edital Faperj grant CNE 2011. VS also thanks the ESF Research
Network Program RGLIS for support. 

%%%%%%%%%%%%%%%%%%%%%%%%%%%%%%%%%%%%%%%%%%%%%%%%

\end{document}